\documentclass[11pt]{article}

\usepackage{amsmath,amssymb,amsfonts,amsthm,bbm}
\usepackage{color,graphicx}
\RequirePackage[OT1]{fontenc}
\RequirePackage[numbers]{natbib}
\usepackage[colorinlistoftodos]{todonotes}

\RequirePackage{amsthm,amsmath,amsfonts,amssymb,amscd,mathrsfs}
\RequirePackage[numbers]{natbib}
\RequirePackage[colorlinks,citecolor=blue,urlcolor=blue]{hyperref}

\graphicspath{ {./figures/} }

\newtheorem{theorem}{Theorem}

\newtheorem{lemma}{Lemma}
\newtheorem{proposition}{Proposition}
\newtheorem{condition}{Condition}

\newtheorem{example}{Example}
\newtheorem{remark}{Remark}

\renewcommand{\P}{\mathbb{P}}

\newcommand{\E}{\mathbb{E}}
\newcommand{\R}{\mathbb{R}}

\newcommand{\G}{\mathscr{G}}

\newenvironment{enumerate*}%

\makeatletter  
\@ifundefined{@currsizeindex} 
  {\RequirePackage{relsize}\let\dolarger\relsize} 
  {\def\dolarger#1{\larger[#1]}} 

\newcommand*\@@bigtimes[2]{\vphantom{\prod} 
  \vcenter{\hbox{\dolarger{4}$\m@th#1\mkern-2mu\times\mkern-2mu$}}} 
\newcommand*\bigtimes{\mathop{\mathpalette\@@bigtimes\relax}\displaylimits} 
\makeatother

\def\E{\mathbb{E}}\def\P{\mathbb{P}}\def\R{\mathbb{R}}\def\1{\mathbbm{1}}

\begin{document}

	\title{On some information-theoretic aspects of \\  non-linear statistical inverse problems}
	\author{Richard Nickl~~~\& Gabriel P. Paternain \\ ~~ \\ \textsc{University of Cambridge}}

\maketitle

\begin{abstract} Results by van der Vaart (1991) from semi-parametric statistics about the existence of a non-zero Fisher information are reviewed in an infinite-dimensional non-linear Gaussian regression setting. Information-theoretically optimal inference on aspects of the unknown parameter is possible if and only if the adjoint of the linearisation of the regression map satisfies a certain range condition. It is shown that this range condition may fail in a commonly studied elliptic inverse problem with a divergence form equation, and that a large class of smooth linear functionals of the conductivity parameter cannot be estimated efficiently in this case. In particular, Gaussian `Bernstein von Mises'-type approximations for Bayesian posterior distributions do not hold in this setting.
\end{abstract}

\section{Introduction}

The study of \textit{inverse problems} forms an active scientific field at the interface of the physical, mathematical and statistical sciences and machine learning. A common setting is where one considers a `forward map' $\mathscr G$ between two spaces of functions, and the `inverse problem' is to recover $\theta$ from the `data' $\G_\theta \equiv \mathscr G(\theta)$. In real-world measurement settings, data is observed \textit{discretely}, for instance one is given point evaluations $\mathscr G(\theta)(X_i)$ of the function $\mathscr G(\theta)$ on a finite discretisation $\{X_i\}_{i=1}^N$ of the domain of $\G_\theta$. Each time a measurement is taken, a statistical error is incurred, and the resulting noisy data can then be described by a statistical regression model $Y_i =\G_\theta(X_i) + \varepsilon_i$, with regression functions $\{\mathscr G_\theta: \theta \in \Theta\}$ indexed by the parameter space $\Theta$. Such  models have been studied systematically at least since C.F.Gauss~\cite{G1809} and constitute a core part of statistical science ever since.

In a large class of important applications, the family of regression maps $\{\G_\theta: \theta \in \Theta\}$ arises from physical considerations and is described by a \textit{partial differential equation} (PDE). The functional parameter $\theta$ is then naturally \textit{infinite- (or after discretisation step, high-) dimensional}, and the map $\theta \mapsto \G_\theta$ is often \textit{non-linear}, which poses challenges for statistical inference. Algorithms for such `non-convex' problems have been proposed and developed in the last decade since influential work by A. Stuart \cite{S10}, notably based on ideas from \textit{Bayesian inference}, where the parameter $\theta$ is modelled by a Gaussian process (or related) prior $\Pi$. The inverse problem is `solved' by approximately computing the posterior measure $\Pi(\cdot|(Y_i, X_i)_{i=1}^N)$ on $\Theta$ by an iterative (e.g., MCMC) method. While the success of this approach has become evident empirically, an objective mathematical framework that allows to give rigorous statistical and computational guarantees for such algorithms in non-linear problems has only emerged more recently. The types of results obtained so far include statistical \textit{consistency and contraction rate} results for posterior distributions and their means, see \cite{MNP21, AN19, GN20} and also \cite{NS17, N17, NR20, GR20, K21}, as well as \textit{computational guarantees} for MCMC based sampling schemes \cite{HSV14, NW20, BN21}. 

Perhaps the scientifically most desirable guarantees are those for `statistical uncertainty quantification' methods based on posterior distributions, and these are notoriously difficult to obtain. Following a programme originally developed by \cite{CN13, CN14, CR15, R17} in classical `direct' regression models, one way to address this issue is by virtue of the so-called \textit{Bernstein-von Mises theorems} which establish asymptotically (as $N \to \infty$) exact Gaussian approximations to posterior distributions. These exploit the precise but delicate machinery from semi-parametric statistics and Le Cam theory (see \cite{vdV98}) and aim at showing that the actions $\langle \psi, \theta \rangle|(Y_i, X_i)_{i=1}^N$ of infinite-dimensional posterior distributions on a well chosen set of test functions $\psi$ converge -- after rescaling by $\sqrt N$ (and appropriate re-centering) -- to fixed normal $\mathcal N(0, \sigma_\theta^2(\psi))$-distributions (with high probability under the data $(Y_i, X_i)_{i=1}^N$). The limiting variance $\sigma_\theta^2(\psi)$ has an information-theoretic interpretation as the \textit{Cram\'er-Rao lower bound} (inverse Fisher information) of the model (see also Section \ref{lbdsec}). Very few results of this type are currently available in PDE settings. Recent progress in \cite{MNP20} (see also related work in \cite{N17, MNP17, NR20, GK20}) has revealed that Bernstein-von Mises theorems may hold true if the PDE underlying $\G_\theta$ has certain analytical properties. Specifically one has to solve `information equations' that involve the `information operator' $D\G_\theta^* D\G_\theta$ generated by the linearisation $D\G_\theta$ of $\G_\theta$ (with appropriate adjoint $D\G_\theta^*$). The results in \cite{N17, MNP20} achieve this for a class of PDEs where a base differential operator (such as the Laplacian, or the geodesic vector field) is attenuated by an unknown potential $\theta$, and where $\psi$ can be any smooth test function.

\smallskip

In the present article we study a different class of elliptic PDEs commonly used to model steady state diffusion phenomena, and frequently encountered as a `fruitfly example' of a non-linear inverse problem in applied mathematics (see the many references in \cite{S10, GN20}). While this inverse problem can be solved in a statistically consistent way (with `nonparametric convergence rates' to the ground truth, see \cite{NVW18, GN20}), we show here that, perhaps surprisingly, semi-parametric Bernstein-von Mises phenomena for posterior distributions of a large class of linear functionals of the relevant `conductivity' parameter \textit{do in fact not hold} for this PDE, not even just locally in a `smooth' neighbourhood of the standard Laplacian. See Theorems \ref{nogood} and \ref{nogood2}, which imply in particular that the inverse Fisher information $\sigma_\theta^2(\psi)$ does not exist for a large class of smooth $\psi$'s. The results are deduced from a theorem of van der Vaart \cite{vdV91} in general statistical models, combined with a thorough study of the mapping properties of $D\G_\theta$ and its adjoint for the PDE considered. Our negative results should help to appreciate the mathematical subtlety underpinning exact Gaussian approximations to posterior distributions in non-linear inverse problems arising with PDEs.

\section{Information geometry in non-linear regression models} \label{infog}

In this section we review some by now classical material on information theoretical properties of infinite-dimensional regular statistical models \cite{vdV91, vdV98}, and develop the details for a general vector-valued non-linear regression model relevant in inverse problems settings. Analogous results could be obtained in the idealised Gaussian white noise model (cf.~Chapter 6 in \cite{GN16}) sometimes considered in the inverse problems literature.

\subsection{Measurement setup}

Let $(\mathcal X, \mathcal A, \lambda)$ be a probability space and let $V$ be a finite-dimensional vector space of fixed finite dimension $p_V \in \mathbb N$ with inner product $\langle \cdot, \cdot \rangle_V$ and norm $|\cdot|_V$. We denote by
$L^\infty(\mathcal X)$  and $L^2(\mathcal X) = L^2_\lambda(\mathcal X,V),$ the bounded measurable, and $\lambda$- square integrable, $V$-valued functions defined on $\mathcal X$ normed by $\|\cdot\|_\infty$ and $\|\cdot\|_{L_\lambda^2(\mathcal X)}$, respectively.  The inner product on $L^2(\mathcal X)$ is denoted by $\langle \cdot, \cdot \rangle_{L^2(\mathcal X)}$. We will also require Hilbert spaces $L^2(P)=L^2(V \times \mathcal X, P)$ of real-valued functions defined on $V \times \mathcal X$ that are square integrable with respect to a probability measure $P$ on the produce space $V \times \mathcal X$, with inner product $\langle \cdot, \cdot \rangle_{L^2(P)}$. 

\smallskip

We will consider a parameter space $\Theta$ that is subset of a (separable) Hilbert space $(\mathbb H, \langle \cdot, \cdot \rangle_{\mathbb H})$ on which measurable `forward maps' 
\begin{equation}\label{fwdG}
\theta \mapsto \mathscr G(\theta)= \G_\theta, \qquad \mathscr G : \Theta \to L^2_\lambda(\mathcal X, V),
\end{equation}
are defined. Observations then arise in a general random design regression setup where one is given jointly i.i.d.~random vectors  $(Y_i, X_i)_{i=1}^N$ of the form
\begin{equation} \label{model}
Y_i = \mathscr G_\theta(X_i) + \varepsilon_i, \quad \varepsilon_i \sim^{i.i.d} \mathcal N(0, I_V),  \quad i=1, \dots, N,
\end{equation}
where the $X_i$'s are random i.i.d.~covariates drawn from law $\lambda$ on $\mathcal X$. We assume that the covariance $I_V$ of each Gaussian noise vector $\varepsilon_i \in V$ is diagonal for the inner product of $V$. [Most of the content of this section is not specific to Gaussian errors $\varepsilon_i$ in (\ref{model}), cf.~Example 25.28 in \cite{vdV98} for discussion.] 

\smallskip

We consider a `tangent space' $H$ at any fixed $\theta \in \Theta$ such that $H$ is a linear subspace of $\mathbb H$ and such that perturbations of $\theta$ in directions $h \in H$ satisfy $\{\theta + sh, h \in H, s \in \R, |s|<\epsilon\} \subset \Theta$ for some $\epsilon$ small enough.  We denote by $\bar H$ the closure of $H$ in $\mathbb H$ and  will regard $\bar H$ itself as a Hilbert space with inner product $\langle \cdot, \cdot \rangle_\mathbb H$. We employ the following assumption in the sequel.
\begin{condition}\label{ass0}
Suppose $\G$ is uniformly bounded $\sup_{\theta \in \Theta}\|\G(\theta)\|_\infty \le U_\G.$ Moreover for fixed $\theta \in \Theta, x \in \mathcal X,$ and every $h \in H$ suppose that $\G_\theta(x)$ is Gateaux-differentiable in direction $h$, that is, for all $x \in \mathcal X$,
\begin{equation}\label{limit}
|\mathscr G(\theta + sh)(x) - \mathscr G(\theta)(x) - s\mathbb I_{\theta}[h](x)|_V=o(s)~~\text{as } s \to 0,
\end{equation}
for some continuous linear operator $\mathbb I_{\theta}: (H, \|\cdot\|_\mathbb H) \to L^2_\lambda(\mathcal X,V)$, and that for some $\epsilon>0$ small enough and some finite constant $B=B(h, \theta)$,
\begin{equation}\label{domin}
\sup_{|s|<\epsilon}\frac{\|\G(\theta+sh) - \G(\theta)\|_\infty}{|s|} \le B.
\end{equation}
\end{condition}

\subsection{The DQM property}

We will now derive the semi-parametric `score' and `information' operators (cf.~\cite{vdV91, vdV98}) in the observational model (\ref{model}). If $P_\theta$ is the law of $(Y_1, X_1)= (\mathscr G(\theta)(X_1)+\varepsilon_1, X_1)$ on $V \times \mathcal X$ then (\ref{model}) is an i.i.d.~statistical model of product laws 
\begin{equation} \label{dmodel}
\mathscr P_N = \{P_\theta^N = \otimes_{i=1}^N P_\theta: \theta \in \Theta\},~~ N \in \mathbb N,
\end{equation}
on $(V \times \mathcal X)^N$, and we can identify all information theoretic properties in terms of the model $\mathscr P = \mathscr P_1= \{P_\theta: \theta \in \Theta\}$ for the coordinate distributions. The model $\mathscr P$ is \textit{differentiable in quadratic mean (DQM)} at $\theta \in \Theta$ along the tangent space $H$ with score operator 
\begin{equation}\label{scoop}
\mathbb A_\theta : H \to L^2(V \times \mathcal X, P_\theta)
\end{equation}
(cf.~(3.2) in \cite{vdV91}) if
 for each path $\theta_{s,h}= \theta+ sh, h \in H,$ we have as $s \to 0$,
\begin{equation}\label{DQM}
\int_{V \times \mathcal X} \Big[\frac{1}{s}\big(dP_{\theta_{s,h}}^{1/2} - dP_\theta^{1/2}\big) - \frac{1}{2} \mathbb A_\theta [h] ~dP_\theta^{1/2} \Big]^2 \to 0
\end{equation}
where $$dP^{1/2}_\theta(y,x) = (2\pi)^{-p_V/4} e^{-|y-\G(\theta)(x)|_V^2/4}dydx,~~ y \in V, x \in \mathcal X,$$ are the square root probability densities of $P_\theta$ with respect to Lebesgue measure on $V \times \mathcal X$.

\begin{theorem}\label{dqmprop}
Assuming Condition \ref{ass0}, the model (\ref{dmodel}) is differentiable in quadratic mean (DQM) at $\theta \in \Theta$ along every path $(\theta + sh: |s|<\epsilon, h \in H)$ with $\epsilon$ small enough. The `score' operator $\mathbb A_\theta:  H \to L^2(V \times \mathcal X, P_\theta)$ is given by
\begin{equation} \label{scop}
\mathbb A_\theta[h](y,x) = \langle y-\mathscr G(\theta)(x),  \mathbb I_{\theta}(h)(x) \rangle_V, ~~ h \in H, ~~(y,x) \in V \times \mathcal X,
\end{equation}
which extends to a continuous linear operator $\mathbb A_\theta : \bar H \to L^2(P_\theta)$.
\end{theorem}
\begin{proof}
Fix $h \in H$. Using that the densities $dP_\theta$ are strictly positive the l.h.s.~in (\ref{DQM}) equals 
\begin{align*}
&\int_{V \times \mathcal X} \Big[\frac{1}{s}\Big(\frac{dP_{\theta_{s,h}}^{1/2}}{dP_\theta^{1/2}} - 1\Big) - \frac{1}{2} \mathbb A_\theta [h] \Big]^2 dP_\theta \\
& = \int_{V \times \mathcal X} \left[ \frac{1}{s}\Big[e^{\langle \frac{y}{2}, \G(\theta_{s,h})(x) - \G(\theta)(x)\rangle_V - \frac{|\G(\theta_{s,h})(x)|_V^2- |\G(\theta)(x)|_V^2}{4}} - 1 \Big] - \frac{1}{2} \mathbb A_\theta [h] \right]^2 dP_\theta(y,x) \\
& = \int_{V \times \mathcal X} \left[\frac{1}{s} \left[e^{f(s)} - 1 - \frac{s}{2} \mathbb A_\theta [h] \right] \right]^2 dP_\theta
\end{align*}
where, for $y,x$ fixed, $$f(s) = \langle \frac{y}{2}, \G(\theta_{s,h})(x) - \G(\theta)(x)\rangle_V - \frac{|\G(\theta_{s,h})(x)|_V^2- |\G(\theta)(x)|_V^2}{4}.$$ Clearly $f(0)=0$ and by Condition \ref{ass0} and the chain rule we have $$f'(0)= \langle \frac{y}{2}, \mathbb I_\theta[h](x)\rangle_V - \frac{\langle \G(\theta)(x), \mathbb I_\theta[h](x)\rangle_V}{2} = \frac{1}{2} \mathbb A_\theta[h](y,x)$$ so that the last integrand converges to zero for every $(y,x) \in V \times \mathcal X$, as $ s \to 0$. By Condition \ref{ass0} and the Cauchy-Schwarz inequality we see that $[e^{f(s)} - 1]/s$ is bounded by a constant multiple of $e^{C|y|_V}, C=C(B, U_\G)<\infty,$ uniformly in $|s|<\epsilon$. Furthermore, again from Condition \ref{ass0}, $$\|\mathbb A_\theta[h]\|_{L^2(P_\theta)} \lesssim [E|Y|_V + U_\G] \|\mathbb I_\theta[h]\|_{L^2_\lambda} \lesssim \|h\|_{\mathbb H}, \text{ and }~E_\theta [e^{C|Y|_V}+  |\mathbb A_\theta [h](Y,X)|]^2 <\infty,$$ so the last limit can be $P_\theta$-integrated by the dominated convergence theorem to give that the last displayed integral converges to zero, verifying the DQM property. The first inequality in the last display also implies that $\mathbb A_\theta$ extends to a continuous linear map from $\bar H$ to $L^2(P_\theta)$.
\end{proof}

\subsection{The adjoint score and information operator}

The bounded linear operator $\mathbb A_\theta : (\bar H, \langle \cdot, \cdot \rangle_{\mathbb H}) \to L^2(V \times \mathcal X, P_\theta)$ has adjoint operator $$\mathbb A_\theta^*:L^2(V \times \mathcal X, P_\theta) \to (\bar H, \langle \cdot, \cdot \rangle_\mathbb H)$$ which satisfies
$$\langle w, \mathbb A_\theta h \rangle_{L^2(P_\theta)} = \langle \mathbb A_\theta^* w, h \rangle_\mathbb H,~\text{ for all}~ w \in L^2(V \times \mathcal X, P_\theta), ~ h \in \bar H.$$ The \textit{information operator} is then defined as 
\begin{equation}\label{infoop}
\mathbb A_\theta^* \mathbb A_\theta:  \bar H \to \bar H.
\end{equation}
Note that the `complexity' of the statistical model enters via the choice of `tangent space' $H$ for which the adjoint is computed, but we suppress this  in the notation.

\smallskip

In the present model the information operator can be entirely described in terms of the operator $\mathbb I_\theta: (H, \langle \cdot, \cdot \rangle_{\mathbb H}) \to L^2_\lambda(\mathcal X, V)$ from Condition \ref{ass0}, and its adjoint $$\mathbb I_\theta^*: L^2_\lambda(\mathcal X, V) \to (\bar H, \langle \cdot, \cdot \rangle_{\mathbb H}).$$

\begin{proposition}
Assuming Condition \ref{ass0} we have 
\begin{equation}\label{infoopid}
\mathbb A_\theta^* \mathbb A_\theta [h] = \mathbb I_\theta^* \mathbb I_\theta[h],~~ \forall h \in H.
\end{equation}
\end{proposition}
\begin{proof}
Writing $\phi$ for the pdf of a $\mathcal N(0,I_V)$ distribution we have from Fubini's theorem, for any $w \in L^2(P_\theta)$,\begin{align*}
\langle \mathbb A_\theta h, w \rangle_{L^2(P_\theta)} & = \int_V \int_\mathcal X \big\langle y-\G_\theta(x), \mathbb I_\theta(h)(x) \big\rangle_V w(y,x) dP_\theta(y,x) \\
&= \int_\mathcal X \Big\langle \mathbb I_\theta(h)(x), \int_V (y-\G_\theta(x)) w(y,x) \phi(y-\G_\theta(x)) dy \Big\rangle_Vd \lambda(x) \\
& = \big\langle \mathbb I_\theta(h), E_\theta[(Y-\G_\theta(X))w(Y,X)|X=\cdot] \big\rangle_{L^2_\lambda} \\
&= \big\langle h, \mathbb I_\theta^*\big[E_\theta[(Y-\G_\theta(X))w(Y,X)|X=\cdot]\big] \big\rangle_{\mathbb H},
\end{align*}
that is, the adjoint $\mathbb A_\theta^* = \mathbb I_\theta^* \circ \mathscr E_\theta$ is the composition of the adjoint $\mathbb I_\theta^*$ of $\mathbb I_\theta$ with the conditional expectation (projection) operator 
\begin{equation}\label{projreg}
\mathscr E_\theta: L^2(P_\theta) \to L^2_\lambda(\mathcal X, V),~~\mathscr E_\theta [w](x) = E_\theta[(Y-\G_\theta(X))w(Y,X)|X=x],~x \in \mathcal X.
\end{equation}
Now for $h \in H$  we see for $\varepsilon \sim \mathcal N(0,I_V)$ and $\lambda$-a.e.~$x \in \mathcal X$,
\begin{align*}
\mathscr E_\theta\big[\mathbb A_\theta[h] \big](x) &= E_\theta\big[(Y-\G_\theta (X))\langle Y-\G_\theta(X), \mathbb I_\theta h(X)\rangle_V |X=x\big] \\
&=  E[\varepsilon \langle \varepsilon, \mathbb I_\theta h(x) \rangle_V] = \mathbb I_\theta h(x).
\end{align*}
and therefore $\mathbb A_\theta^* \mathbb A_\theta [h]=\mathbb I_\theta^* \mathscr E_\theta[\mathbb A_\theta[h]] = \mathbb I_\theta^* \mathbb I_\theta[h],$ completing the proof.
\end{proof}

One can think of $\mathscr E_\theta$ in the previous proof as a projection onto the `space of residuals' of the regression equation (\ref{model}), which vanishes in the representation of the information operator (\ref{infoopid}). In particular the model (\ref{model}) is LAN (locally asymptotically normal) for LAN-norm $\|\cdot\|_{LAN}$ arising from LAN inner product 
\begin{equation}\label{lanip}
\langle h_1, h_2 \rangle_{LAN} := \langle \mathbb I_\theta h_1, \mathbb I_\theta h_2 \rangle_{L^2_\lambda} = \langle \mathbb A_\theta h_1, \mathbb A_\theta h_2 \rangle_{L^2(P_\theta)},~~ h_1, h_2 \in \bar H.
\end{equation}
\begin{proposition}\label{LAN}
Let $D_N \equiv (Y_i, X_i)_{i=1}^N \sim P_\theta^N$ arise from model (\ref{model}) for some $\theta \in \Theta$ and suppose Condition \ref{ass0} holds. Then the likelihood ratio process satisfies
$$\log \frac{dP^N_{\theta+h/\sqrt N}}{dP_\theta}(D_N)  \to_{N \to \infty}^d \mathcal N\Big(-\frac{1}{2}\|h\|^2_{LAN}, \|h\|_{LAN}^2\Big),~~h \in H.$$
\end{proposition}

The proof follows from Theorem \ref{dqmprop} in conjunction with Lemma 25.14 in \cite{vdV98} (and the central limit theorem). This in particular justifies the use of the terminology `information operator' for $\mathbb I_\theta^* \mathbb I_\theta$ instead of $\mathbb A_\theta^* \mathbb A_\theta$.

\medskip

In what is to follow, the range of the adjoint score operator $\mathbb A_\theta^*$ will play a crucial role, and we wish to record a few preparatory remarks here. By what precedes that range  equals
\begin{equation} \label{range00}
R(\mathbb A_\theta^*)=\big\{\psi =  \mathbb I_\theta^* \mathscr E_\theta w,~\text{ for some } w \in L^2(P_\theta)\big\}.
\end{equation}
where $\mathscr E_\theta$ is from (\ref{projreg}). Since $\mathscr E_\theta$ maps $L^2(P_\theta)$ into $L_\lambda^2$, a fortiori any $\psi \in R(\mathbb A_\theta^*)$ has to satisfy
\begin{equation} \label{range01}
\psi \in R(\mathbb I_\theta^*)=\{\psi = \mathbb I_\theta^* h,~\text{ for some } h \in L_\lambda^2(\mathcal X, V)\},
\end{equation}
so $R(\mathbb A_\theta^*) \subset R(\mathbb I_\theta^*)$. Likewise, taking $w(y,x)= \langle y-\G(\theta)(x), h(x) \rangle_V \in L^2(P_\theta)$ we can realise (arguing as in the proof of the last proposition) any $h \in L_\lambda^2(\mathcal X)$ as a $\mathscr E_\theta w=h$ and so if $\psi \in R(\mathbb I_\theta^*)$ then $\psi \in R(\mathbb A_\theta^*)$, too. We conclude that 
\begin{equation}\label{range0}
R(\mathbb I_\theta^*) = R(\mathbb A_\theta^*).
\end{equation}

\subsection{Lower bounds for estimation of functionals}\label{lbdsec}

Suppose the problem is to estimate a \textit{linear} functional $\Psi : \Theta \to \R$ of the unknown parameter $\theta$. Let $$\mathscr P_H:=\{w= \mathbb A_\theta(h): h \in H\} \subset L^2(V \times \mathcal X, P_\theta)$$ denote the tangent space of the model $\mathscr P$ induced by $H$. Suppose further we can find $\tilde \psi_\theta \in L^2(P_\theta)$ (the `efficient influence function') s.t.
\begin{equation}\label{riesz}
\Psi(h)= \langle  \tilde \psi_\theta, \mathbb A_\theta h \rangle_{L^2(P_\theta)},~~~ h \in H.
\end{equation}
If such $\tilde \psi_\theta$ exists we can always take it to belong to the closure $\overline{\mathscr P_H}$ of $\mathscr P_H$ in $L^2(P_\theta)$ (simply by $L^2(P_\theta)$-projection onto $\overline{\mathscr P_H}$, if necessary).  A lower bound for the optimal efficient asymptotic variance for $\sqrt N$-consistent estimators of $\Psi(\theta)$ over the model $\{\theta + h/\sqrt N, h \in H\}$ is then given by
\begin{equation}\label{invfi}
\sup_{0\neq w  \in \mathscr P_H} \frac{\langle \tilde \psi_\theta, w \rangle_{L^2(P_\theta)}^2}{\langle w , w \rangle_{L^2(P_\theta)}} = \|\tilde \psi_\theta\|_{L^2(P_\theta)}^2,
\end{equation}
with equality holding in view of $\tilde \psi_\theta \in \overline{\mathscr P_H}$ and the Cauchy-Schwarz inequality. Specifically by Theorem 25.21 in \cite{vdV98} one has
\begin{equation}\label{locmin}
\liminf_{N\to \infty} \inf_{\tilde \psi_N: (V \times \mathcal X)^N \to \R} \sup_{h \in H, \|h\|_{\mathbb H} \le 1/\sqrt N} N E_{\theta+h}^N(\tilde \psi_N - \Psi(\theta+h))^2 \ge \|\tilde \psi_\theta\|_{L^2(P_\theta)}^2.
\end{equation}
If the functional is of the form $\Psi(h) = \langle \psi, h \rangle_{\mathbb H}$ for some fixed test function $\psi$, and if $\mathbb A_\theta^*$ is the adjoint of $\mathbb A_\theta$ from the previous subsection, the requirement (\ref{riesz}) can be written as 
\begin{equation}\label{rieszh}
\langle \psi, h \rangle_{\mathbb H}=\langle  \tilde \psi_\theta, \mathbb A_\theta h \rangle_{L^2(P_\theta)} = \langle \mathbb A_\theta^* \tilde \psi_\theta, h \rangle_{\mathbb H},~~~ h \in H,
\end{equation}
and hence reduces to $\psi = \mathbb A_\theta^* \tilde \psi_\theta$ for some $\tilde \psi_\theta \in  L^2(P_\theta)$, that is, $\psi \in R(\mathbb A_\theta^*)$ from (\ref{range00}).

\subsection{Non-existence of $\sqrt N$-consistent estimators of linear functionals}

 Arguing along the traditional lines of the proof of the Cramer-Rao inequality, the inverse of
\begin{equation}
i_{\theta, h, \psi} := \frac{\|\mathbb A_\theta h \|^2_{L^2(P_\theta)}}{\langle  \psi,  h \rangle_{\mathbb H}^2}
\end{equation}
provides an apriori lower bound for the variance of any estimator $\hat \Psi$ of $\Psi(\theta)=\langle \psi, \theta \rangle_{\mathbb H}$ that is unbiased (i.e., satisfies $E_\theta \hat \Psi = \Psi(\theta)$) for all $\theta$ in the one-dimensional model $\{\theta + sh: |s|<\epsilon\}$. The \textit{efficient} Fisher information for estimating $\Psi$ optimally for all elements $h \in H$ of the tangent space is then given by 
\begin{equation}\label{efffish}
i_{\theta, H, \psi} := \inf_{h \in H, \langle  \psi,  h \rangle_{\mathbb H} \neq 0} \frac{\|\mathbb A_\theta h \|^2_{L^2(P_\theta)}}{\langle  \psi,  h \rangle_{\mathbb H}^2}.
\end{equation}
  Note that when $\psi=\mathbb A_\theta^* \tilde \psi_\theta$ is in the range of $\mathbb A_\theta^*$ then we can rewrite the last number as
\begin{equation}\label{fishy}
\inf_{h \in H, \langle \psi, h \rangle_{\mathbb H} \neq 0} \frac{\|\mathbb A_\theta h \|^2_{L^2(P_\theta)}}{\langle  \mathbb A_\theta^* \tilde \psi_\theta,  h \rangle_{\mathbb H}^2}=\inf_{h \in H, \langle \psi, h \rangle_{\mathbb H} \neq 0} \frac{\|\mathbb A_\theta h \|^2_{L^2(P_\theta)}}{\langle  \tilde \psi_\theta,  \mathbb A_\theta h \rangle_{L^2(P_\theta)}^2}.
\end{equation}
Since $\psi \in R(\mathbb A_\theta^*)$ is orthogonal on $\rm{ker}(\mathbb A_\theta)$, using also (\ref{invfi}) we thus arrive at
\begin{equation}\label{lbd}
 \|\tilde \psi_\theta\|_{L^2(P_\theta)}^2 = \sup_{h  \in H, \mathbb A_\theta h \neq 0} \frac{\langle \tilde \psi_\theta, \mathbb A_\theta h \rangle_{L^2(P_\theta)}^2}{\langle \mathbb A_\theta h  , \mathbb A_\theta h  \rangle_{L^2(P_\theta)}} = i^{-1}_{\theta, H, \psi},
\end{equation}
explaining the relationship to the best asymptotic variance in (\ref{locmin}). 
 
 An important observation of van der Vaart (Theorem 4.1 in \cite{vdV91}) is that a necessary and sufficient condition for the Fisher information for estimating $\Psi(\theta)=\langle \theta, \psi \rangle_\mathbb H$ to be non-zero is that $\psi$ indeed lies in the range of $\mathbb A_\theta^*$.  

\begin{theorem}\label{aad}  For $\theta \in \Theta$ and tangent space $H$, let $i_{\theta,H, \psi}$ be the efficient Fisher information (\ref{efffish}) for estimating the functional $\Psi(\theta)=\langle \theta, \psi \rangle_{\mathbb H}, \psi \in \bar H$. Then $i_{\theta, H, \psi}>0$ if and only if $\psi \in R(\mathbb I_\theta^*)$.
\end{theorem}

If $\psi \in R(\mathbb I_\theta^*)$ then positivity $i_{\theta, H, \psi}>0$ follows directly from (\ref{range0}), (\ref{fishy}) and the Cauchy-Schwarz inequality. The converse is slightly more involved -- we include a proof in Section \ref{spec} below for the case most relevant in inverse problems when the information operator $\mathbb I_\theta^* \mathbb I_\theta$ from (\ref{infoopid}) is \textit{compact} on $\bar H$ (see after Proposition \ref{adjform} below for the example relevant here). 

\smallskip

It follows that if $\psi \notin R(\mathbb I_\theta^*)$ then $\Psi(\theta)$ cannot be estimated at $\sqrt N$-rate in minimax risk. 

\begin{theorem}\label{n00}
Consider estimating a functional $\Psi(\theta)=\langle \psi, \theta \rangle_{\mathbb H}, \psi \in \bar H,$ based on i.i.d.~data $(Y_i, X_i)_{i=1}^N$ in the model (\ref{model}) satisfying Condition \ref{ass0} for some $\theta \in \Theta$ and tangent space $H$. Suppose  $i_{\theta, H, \psi}=0$. Then 
\begin{equation}\label{locmininf}
\liminf_{N \to \infty} \inf_{\tilde \psi_N: (V \times \mathcal X)^N \to \R} \sup_{h \in H, \|h\|_{\mathbb H} \le 1/\sqrt N} N E_{\theta+h}^N(\tilde \psi_N - \Psi(\theta+h))^2 =\infty.
\end{equation}
\end{theorem}

The last theorem can be proved following the asymptotic arguments leading to the proof of (\ref{locmin}) in Theorem 25.21 in \cite{vdV98}.  A proof that follows more directly from the preceding developments is as follows: Augment the observation space to include measurements $(Z_i, Y_i, X_i)_{i=1}^N \sim \bar P_\theta^N$ where the $Z_i \sim^{iid} \mathcal N(\langle \theta, \psi \rangle_\mathbb H, \sigma^2)$ are independent of the $(Y_i, X_i)$'s, and where $\sigma^2$ is known but arbitrary. The new model $\bar{\mathscr P_N} = \{\bar P_\theta^N: \theta \in \Theta\}$ has `augmented' LAN norm from (\ref{lanip}) given by $$\|\bar{\mathbb A}_\theta h\|_{L^2(\bar P_\theta)}^2 = \|\mathbb A_\theta h \|^2_{L^2(P_\theta)} +  \sigma^{-2}\langle \psi, h \rangle_{\mathbb H}^2,~~ h \in \bar H,$$
as can be seen from a standard tensorisation argument for independent sample spaces and the fact that a $\mathcal N(\langle \theta, \psi \rangle_\mathbb H, \sigma^2)$ model has LAN `norm' $\sigma^{-2}\langle \psi, h \rangle_{\mathbb H}^2$, by a direct calculation with Gaussian densities. In particular the efficient Fisher information from (\ref{efffish}) for estimating $\langle \psi, \theta\rangle_\mathbb H$ from the augmented data  is now of the form
$$ \bar i_{\theta, H, \psi} = \inf_{h} \frac{\|\mathbb A_\theta h \|^2_{L^2(P_\theta)} +  \sigma^{-2}\langle \psi, h \rangle_{\mathbb H}^2}{\langle  \psi,  h \rangle_{\mathbb H}^2} = i_{\theta, H, \psi} + \sigma^{-2} = \sigma^{-2}>0.
$$ 
Note next that \textit{mutatis mutandis}, (\ref{invfi}), (\ref{locmin}), (\ref{lbd}) all hold in the augmented model $\bar {\mathscr P_N}$ with score operator $\bar {\mathbb A}_\theta$ and tangent space $H$, and that the linear functional $\Psi(\cdot)=\langle \psi, \cdot\rangle_{\mathbb H}$ now verifies (\ref{riesz}) as it is continuous on $H$ for the $\|\bar{\mathbb A}_\theta[\cdot]\|_{L^2(\bar P_\theta)}$-norm so that we can invoke the Riesz representation theorem to the effect that $$\Psi(h) = \langle \bar{\mathbb A}\tilde h, \bar{\mathbb A}h\rangle_{L^2(\bar P_\theta)},~~h \in H, \text{ and some } \tilde \psi_\theta =  \bar{\mathbb A}\tilde h \in \overline{(\bar {\mathscr P})_H}.$$  Thus the asymptotic minimax theorem in the augmented model gives
\begin{equation}\label{auxrisk}
\liminf_{N \to \infty} \inf_{\bar \psi_N: (\R \times V \times \mathcal X)^N \to \R} \sup_{h \in H, \|h\|_{\mathbb H} \le 1/\sqrt N} N E_{\theta+h}^N(\bar \psi_N - \Psi(\theta+h))^2 \ge \bar i^{-1}_{\theta, H, \psi} = \sigma^2
\end{equation}
for estimators $\bar \psi$ based on the more informative data. The asymptotic local minimax risk in (\ref{locmininf})  exceeds the quantity in the last display, and letting $\sigma^2 \to \infty$ implies the result.

\section{Application to a divergence form PDE}\label{apppde}

The results from the previous section describe how in a non-linear regression model (\ref{model}) under Condition \ref{ass0}, the possibility of $\sqrt N$-consistent estimation of linear functionals $\Psi(\theta)=\langle \psi, \theta \rangle_{\mathbb H}$ essentially depends on whether $\psi$ lies in the range of $\mathbb I_\theta^*$. A sufficient condition for this is that $\psi$ lies in the range of the information operator $\mathbb A_\theta^* \mathbb A_\theta=\mathbb I_\theta^* \mathbb I_\theta$, and the results in \cite{MNP20} show that the lower bound in (\ref{locmin}) can be attained by concrete estimators in this situation. The general theory was shown to apply to a class of PDEs of Schr\"odinger type \cite{N17, MNP20} and to non-linear $X$-ray transforms \cite{MNP20, MNP17}, with smooth test functions $\psi \in C^\infty$.

\smallskip

We now exhibit a PDE inverse problem where the range constraint from Theorem \ref{aad} fails, fundamentally limiting the possibility of efficient $\sqrt N$-consistent estimation of `nice' linear functionals. In particular we will show that, unlike for the Schr\"odinger type equations considered in \cite{N17, MNP20}, for this PDE the inverse Fisher information $\sigma_\theta^2(\psi)$ does not exist for a large class of functionals $\Psi(\theta)=\langle \theta, \psi \rangle_{L^2}$, including generic examples of \textit{smooth non-negative} $\psi \in C^\infty$. This implies in particular the non-existence of a `functional' Bernstein-von Mises phenomenon that would establish asymptotic normality of the posterior distribution of the process $\{\langle \theta, \psi \rangle_{L^2}: \psi \in C^\infty\}$ (comparable to the ones obtained in \cite{CN13, CN14, N17}).

\subsection{Basic setting}

Let $\mathcal O \subset \mathbb R^d$ be a bounded smooth domain with boundary $\partial \mathcal O$ and, for convenience, of unit volume $\lambda(\mathcal O)=1$, where $\lambda$ is Lebesgue measure. Denote by $C^\infty(\mathcal O)$ the set of all smooth real-valued functions on $\mathcal O$ and by $C^\infty_0(\mathcal O)$ the subspace of such functions of compact support in $\mathcal O$. Let $L^2=L^2_\lambda(\mathcal O)$ be the usual Hilbert space with inner product $\langle \cdot, \cdot \rangle_{L^2}$. The $L^2_\lambda$-Sobolev spaces $H^\beta=H^\beta(\mathcal O)$ of order $\beta \in \mathbb N$ are also defined in the standard way, as are the spaces $C^\beta(\mathcal O)$ that have all partial derivatives bounded and continuous up to order $\beta$.

\smallskip

For a \textit{conductivity} $\theta \in C^\infty(\mathcal O)$,  \textit{source} $f \in C^\infty(\mathcal O)$ and \textit{boundary temperatures} $g \in C^\infty(\partial \mathcal O)$, consider solutions $u=u_\theta=u_{\theta, f,g}$ of the PDE
\begin{align}\label{PDE}
\nabla \cdot (\theta \nabla u) &=f ~\text{in } \mathcal O,\\
u &=g \text{ on } \partial \mathcal O \notag.
\end{align}
Here $\nabla, \Delta, \nabla \cdot$ denote the gradient, Laplace and divergence operator, respectively. We ensure ellipticity by assuming  $\theta \ge \theta_{min}>0$ throughout $\mathcal O$. 

\smallskip

We write $\mathcal L_\theta=\nabla \cdot (\theta \nabla (\cdot))$ for the `divergence form' operator featuring on the l.h.s.~in (\ref{PDE}).  A unique solution $u_\theta \in C^\infty(\mathcal O)$ to (\ref{PDE}) exists (e.g., Theorem 8.3 and Corollary 8.11 in \cite{GT98}). The operator $\mathcal L_\theta$ has an inverse integral operator 
\begin{equation} \label{invlap}
V_\theta: L_\lambda^2(\mathcal O) \to H^2(\mathcal O) \cap \{h_{|\partial \mathcal O}=0\}
\end{equation} for Dirichlet boundary conditions, that is, it satisfies $V_\theta[f]=0$ at $\partial \mathcal O$ and $\mathcal L_\theta V_\theta [f] = f$ on $\mathcal O$ for all $f \in L_\lambda^2(\mathcal O)$. Moreover the operator $V_f$ is self-adjoint on $L_\lambda^2(\mathcal O)$.  One further shows that whenever $f \in H^2(\mathcal O)$ satisfies $f_{|\partial \mathcal O}=0$, then $ V_\theta \mathcal L_\theta [f]= f$. These standard facts for elliptic PDEs can be proved, e.g., as in Section 5.1 in \cite{Tay11} or Chapter 2 in \cite{LM72}.

\smallskip

To define the `forward map' $\G$ we consider a model $\Theta$ of conductivities arising as a $H^\beta$-neighbourhood of the standard Laplacian of radius $\eta>0$, specifically
\begin{equation}\label{bigtheta}
\Theta  = \Big\{\theta \in C^\infty(\mathcal O), \inf_x\theta(x)>\frac{1}{2}, \theta_{|\partial \mathcal O}=1: \|\theta-1\|_{H^\beta(\mathcal O)} < \eta \Big\},~~\beta>1+d.
\end{equation}
The inverse problem is to recover $\theta$ from solutions 
\begin{equation}\label{forfad}
\G: \Theta \to L_\lambda^2(\mathcal O), ~~~\G(\theta) \equiv u_\theta
\end{equation}
of (\ref{PDE}) where we emphasise that $f,g,$ as well as $\theta_{|\partial \mathcal O}$ are assumed to be \textit{known} (see also Remark \ref{normald}). The particular numerical choices $1=\theta_{|\partial \mathcal O}$ and $1/2 =\theta_{min}$ are made for notational convenience. For independent $\varepsilon_i \sim^{iid} \mathcal N(0,1),~ X_i \sim^{iid} \lambda$ we then observe data 
\begin{equation} \label{modelp}
(Y_i, X_i)_{i=1}^N \in (\R \times \mathcal O)^N \sim P_\theta^N,~~Y_i = u_\theta(X_i) + \varepsilon_i, 
\end{equation}
from model (\ref{model}). Note that unlike in statistical `Calder\'on problems' \cite{AN19}, we measure $u_\theta$ throughout the entire domain $\mathcal O$. Before we take a closer look at the \textit{local} information geometry of the map $\G$ arising from the PDE (\ref{PDE}), let us first give conditions under which the problem of inferring $\theta$ from $(Y_i, X_i)_{i=1}^N$ in (\ref{modelp}) has a consistent solution.

\subsection{Global injectivity and model examples}

Under suitable constellations of $f,g$ in (\ref{PDE}), the non-linear map $\theta \mapsto u_\theta$ can be injective, and `stability' properties of $\G$ are well studied at least since \cite{R81}, we refer to the recent contributions \cite{ABD17, NVW18, GN20} and the many references therein. For instance one can show:
\begin{proposition}\label{oldboy}
Let $\theta_1, \theta_2 \in C^\infty(\mathcal O)$ be conductivities such that $\|\theta_i\|_{C^1} \le B$, $\theta_1=\theta_2$ on $\partial \mathcal O$, and denote by $u_{\theta_i}$ the corresponding solutions to (\ref{PDE}). Assume 
\begin{equation} \label{lowbd}
\inf_{x \in \mathcal O}\big[\Delta u_{\theta}(x) + \mu |\nabla u_{\theta}(x)|_{\R^d}^2 \big] \ge c_0>0
\end{equation} 
holds for $\theta=\theta_1$ and some $\mu>0$. Then we have for some $C=C(B, \mu, c_0, \mathcal O)>0$,
\begin{equation}\label{globstab}
\|\theta_1 - \theta_2\|_{L^2} \le C \|u_{\theta_1} - u_{\theta_2}\|_{H^2}.
\end{equation}
\end{proposition}

Based on (\ref{globstab}) one can show (see \cite{GN20, NVW18}) that we can recover $\theta$ in $L^2$-loss by some estimator $\hat \theta =\hat \theta((Y_i, X_i)_{i=1}^N )$ at a `nonparametric rate' $\|\hat \theta - \theta\|_{L^2(\mathcal O)} =O_{P_\theta^N}(N^{-\gamma})$ for some $0<\gamma<1/2$, uniformly in $\Theta$. We wish to study here inference on  linear functionals $$\Psi(\theta) = \langle \psi, \theta \rangle_{L^2(\mathcal O)},~~ \psi \in C^\infty_0(\mathcal O).$$ As we can bound the `plug-in' estimation error $|\langle \psi, \theta - \hat \theta\rangle_{L^2}|$ by $\|\hat \theta-\theta\|_{L^2}$, the convergence rate $N^{-\gamma}$ carries over to estimation of $\Psi$. Nevertheless we will show that there are fundamental limitations for \textit{efficient} inference on $\Psi$ at the `semi-parametric' rate ($\gamma=1/2$). This will be illustrated with two model examples for which the `injectivity' condition (\ref{lowbd}) can be checked.

\smallskip

\begin{example}{(No critical points). } \label{ex1}\normalfont
In (\ref{PDE}), take
 \begin{equation}\label{repex}
 f=2,~~ g=\frac{|\cdot|_{\R^d}^2-1}{d}
 \end{equation}
 Then for the standard Laplacian $\theta=1$ we have $u_1=g$ on $\bar {\mathcal O}, \Delta u_1=2$, and hence $\nabla u_1 = 2x/d$, which satisfies $\inf_{x \in \mathcal O}|\nabla u_1(x)|_{\R^d} \ge c>0$ 
 for any domain $\mathcal O \subset \R^d$ \textit{separated away from the origin}. This lower bound extends to 
\begin{equation} \label{gradlb}
\inf_{\theta \in \Theta} \inf_{x \in \mathcal O}|\nabla u_\theta(x)|_{\R^d} \ge c_\nabla>0
\end{equation}
for $\eta$ small enough in (\ref{bigtheta}), by perturbation: arguing as in (\ref{diffp}) below and from standard elliptic regularity estimates (Lemma 23 in \cite{NVW18} and as in (\ref{ellipt})), we have for $b>1+d/2, \beta>b+d/2$ (s.t.~$H^\beta \subset C^b$)
\begin{align}\label{pertubio}
\|u_\theta - u_1\|_{C^1} &\lesssim \|V_1[\nabla \cdot [(\theta-1) \nabla u_\theta]\|_{H^{b}}  \lesssim \|(\theta-1)\nabla u_\theta\|_{H^{b-1}} \notag \\
& \lesssim \|\theta-1\|_{H^{b-1}} \|u_\theta\|_{C^{b}} \le \|\theta-1\|_{H^\beta} \|u_\theta\|_{H^\beta}  < C\eta.
\end{align}
In view of $\sup_{\theta \in \Theta}\|\Delta u_\theta\|_\infty<\infty$ and (\ref{gradlb}), condition (\ref{lowbd}) is verified for $\mu$ large enough and all $\theta \in \Theta$.
\end{example}

The situation in Example \ref{ex1} where the gradient $\nabla u_\theta$ never vanishes is somewhat atypical, and one may expect $u_\theta$ to possess a \textit{finite} number of isolated critical points $x_0$ (where $\nabla u_\theta(x_0)$ vanishes), see, e.g., \cite{ABD17} and references therein. The next example encompasses a prototypical such situation with an interior minimum. See also Remark \ref{intrcrit} for the case of a saddle point. Further examples with more than one critical point are easily constructed, too. 

\begin{example}{(Interior minimum). } \label{ex2} \normalfont
Consider the previous example where now $\mathcal O$ is the unit disk in $\R^2$ centred at the origin. In other words, in (\ref{PDE}) we have $f=2$ and $g_{|\partial \mathcal O}=0$, corresponding to a classical Dirichlet problem with source $f$. In this case $u_1$ takes the same form as in the previous example but now has a gradient $\nabla u_1=x$ that vanishes at the origin $0 \in \R^2$, corresponding to the unique minimum of $u_1$ on $\mathcal O$. The injectivity condition (\ref{lowbd}) is still satisfied for all $\theta \in \Theta$ simply since (\ref{PDE}) implies $$0<2=\theta \Delta u_\theta + \nabla \theta \cdot \nabla u_\theta~~\text{on } \mathcal O,$$ so that either $\Delta u_{\theta} \ge 1/(2\|\theta\|_\infty)$ or $|\nabla u_\theta(x)|_{\R^d} \ge 1/(2\|\theta\|_{C^1})$ has to hold on $\mathcal O$. In this example, the constraints that $\eta$ be small enough as well as that $\theta_1=\theta_2$ on $\partial \mathcal O$ in Proposition \ref{oldboy} can in fact be removed, see Lemma 24 in \cite{NVW18}.
\end{example}

\subsection{The score operator and its adjoint}

To connect to Section \ref{infog} let us regard $\Theta$ from (\ref{bigtheta}) as a subset of the Hilbert space $\mathbb H = L_\lambda^2(\mathcal O)$, and take $\G(\theta)$ from (\ref{forfad}); hence we set $\mathcal X= \mathcal O, V=\R$, $\lambda=dx$ (Lebesgue measure). 

\smallskip

 As `tangent space' $H \subset \mathbb H$ we take all smooth perturbations of $\theta$ of compact support; \begin{equation} \label{tangentspace}
H = C^\infty_0(\mathcal O),
\end{equation}
so that the paths $\theta_{s,h}= \theta+sh, ~\theta \in \Theta, h \in H,$ lie in $\Theta$ for all $s \in \R$ small enough. The closure $\bar H$ of $H$ for $\|\cdot\|_\mathbb H$ equals $\bar H = \mathbb H=L_\lambda^2(\mathcal O).$ We now check Condition \ref{ass0}, restricting to $d \le 3$ to expedite the proof.

\begin{theorem}\label{lanell}
Assume $d \le 3$. Let $\Theta$ be as in (\ref{bigtheta}) and let the tangent space $H$ be as in (\ref{tangentspace}). The forward map $\theta \mapsto  \G(\theta)$ from (\ref{forfad}) satisfies Condition \ref{ass0} for every $\theta \in \Theta$, with uniform bound $U_\G =U_\G(\|g\|_\infty, \|f\|_\infty)$ and with
\begin{equation}\label{redscop}
\mathbb I_\theta(h) \equiv -V_\theta[\nabla \cdot (h \nabla u_\theta)],~~ h \in H.
\end{equation}
In particular $\mathbb I_\theta$ extends to a bounded linear operator on $\mathbb  H$.
 \end{theorem}
\begin{proof}
We can represent the solutions $u_\theta$ of (\ref{PDE}) by a Feynman-Kac type formula as
\begin{equation}
u_\theta(x) = \E^x g(X_{\tau_\mathcal O})  -\E^x\int_0^{\tau_\mathcal O} f(X_s)ds,~~x \in \mathcal O,
\end{equation}
where $(X_s:s \ge 0)$ is a Markov diffusion process started at $x \in \mathcal O$ with infinitesimal generator $\mathcal L_\theta/2$, law $\P^x=\P^x_\theta$, and exit time $\tau_\mathcal O$ from $\mathcal O$, see Theorem 2.1 on p.127 in \cite{F85}. As in the proof of Lemma 20 in \cite{NVW18} one bounds $\sup_{x \in \mathcal O}\E^x \tau_{\mathcal O}$ by a constant that depends only on $\mathcal O, \theta_{min}$, and we conclude from the last display that therefore
\begin{equation}\label{envel}
\|u_\theta\|_\infty \le \|g\|_\infty + \|f\|_\infty \sup_{x \in \mathcal O}\E^x \tau_\mathcal O <\infty
\end{equation}
so that the bound $U_\G$ for $\G$ required in Condition \ref{ass0} follows.

\smallskip

We will repeatedly use the following elliptic regularity estimates 
\begin{equation} \label{ellipt}
\|V_\theta [h]\|_{\infty} \leq c_0 \|V_\theta [h]\|_{H^2} \le c_1 \|h\|_{L^2},~~\|u_\theta\|_{H^2} \le c_2,
\end{equation}
with constants $c_0=c_0(\mathcal O), c_1=c_1(\theta_{min}, \mathcal O, \beta, \eta), c_2=c_2(U_\G, \|f\|_{L^2}, \|g\|_{H^2}, \theta_{min}, \mathcal O, \beta, \eta)$ that are \textit{uniform} in $\theta \in \Theta$. The first inequality in (\ref{ellipt}) is just the Sobolev imbedding. The second follows from Lemma 21 in \cite{NVW18}, noting also that $\sup_{\theta \in \Theta}\|\theta\|_{C^1} \le C(\beta, \eta, \mathcal O)$ by another Sobolev imbedding $H^\beta \subset C^1$. The final inequality in (\ref{ellipt}) follows from Theorem 8.12 in \cite{GT98} and (\ref{envel}).

\smallskip

To verify (\ref{domin}) notice that the difference $u_{\theta + sh}- u_\theta$ solves (\ref{PDE}) with $g=0$ and appropriate right hand side, specifically we can write 
\begin{equation}\label{diffp}
\G(\theta+sh) - \G(\theta) = -sV_\theta[\nabla \cdot (h \nabla u_{\theta+sh})],~~ h \in H,
\end{equation}
for $|s|$ small enough. Then (\ref{domin}) follows from (\ref{ellipt}) since
\begin{align*}
\|V_\theta[\nabla \cdot (h \nabla u_{\theta+sh})]\|_{\infty} &\lesssim \|\nabla \cdot (h \nabla u_{\theta+sh}))\|_{L^2}  \lesssim \|h \nabla u_{\theta+sh}\|_{H^1} \lesssim \|h\|_{C^1} \sup_{\theta \in \Theta}\|u_{\theta}\|_{H^2} \le B<\infty.
\end{align*}

We will verify (\ref{limit}) by establishing a stronger `$\|\cdot\|_\infty$-norm' differentiability result: fix $\theta \in \Theta$ and any $h \in H$ such that $\theta+h \in \Theta$. Denote by $D\G_\theta[h]$ the solution $v=v_h$ of the PDE 
\begin{align*}
\nabla \cdot (\theta \nabla v) &=-\nabla \cdot (h \nabla u_\theta) ~\text{on } \mathcal O, \\
v &=0 \text{ on } \partial \mathcal O
\end{align*}
where $u_\theta$ is the given solution of the original PDE (\ref{PDE}). Then the function $w_h = u_{\theta+h}-u_\theta - D\G_\theta[h]$ solves the PDE
\begin{align*}
\mathcal L_{\theta+h}w_h &= -\nabla \cdot (h  \nabla  V_{\theta}[\nabla \cdot (h \nabla u_\theta)) ~\text{on } \mathcal O, \\
w_h &=0 \text{ on } \partial \mathcal O.
\end{align*}
As a consequence, applying (\ref{ellipt}) and standard inequalities repeatedly we have
\begin{align} \label{linapprox}
\|u_{\theta+h}-u_\theta - D\G_\theta[h]\|_\infty &= \|V_{\theta+h}\big[\nabla \cdot (h  \nabla  V_{\theta}[\nabla \cdot (h \nabla u_\theta))\big]\|_{\infty} \notag \\
& \lesssim \left\|\nabla \cdot (h  \nabla  V_{\theta}[\nabla \cdot (h \nabla u_\theta))\right\|_{L^2} \notag \\
&\lesssim \|h\|_{C^1}  \|V_{\theta}[\nabla \cdot (h \nabla u_\theta))]\|_{H^2}  \notag \\
& \lesssim \|h\|_{C^1} \|\nabla \cdot (h \nabla u_\theta)\|_{L^2}\notag  \\
& \lesssim \|h\|_{C^1}^2 \|u_\theta\|_{H^2}= O(\|h\|^2_{C^1}). 
\end{align}
In particular $D\G_\theta[sh] = \mathbb I_\theta[sh]$ is the linearisation of the forward map $\theta \mapsto \G(\theta)=u_\theta$ along any path $\theta+sh, |s| >0, h \in H$.  Finally, by duality, self-adjointness of $V_\theta$ and the divergence theorem (Prop. 2.3 on p.143 in \cite{Tay11}) we can bound for every $h \in H$,
\begin{align*}
\|\mathbb I_\theta h \|_{L^2} &= \sup_{\|\phi\|_{L^2} \le 1} \Big|\int_\mathcal O \phi V_\theta[\nabla \cdot (h \nabla u_\theta)] \Big| = \sup_{\|\phi\|_{L^2} \le 1} \Big|\int_\mathcal O \nabla V_\theta [\phi]  \cdot h\nabla u_\theta] \Big| \\
&\lesssim \sup_{\|\phi\|_{L^2} \le 1} \|V_\theta [\phi]\|_{H^1} \|h\|_{L^2} \|u_\theta\|_{C^1} \lesssim \|h\|_{L^2}
\end{align*}
using also (\ref{ellipt}) and that $\|u_\theta\|_{C^1}<\infty$ (here for fixed $\theta$) as $u_\theta$ is smooth. By continuity and since $H$ is dense in $L^2_\lambda=\mathbb H$, we can extend $\mathbb I_\theta$ to a bounded linear operator on $\mathbb H$, completing the proof.
\end{proof}

Theorem \ref{dqmprop} gives the score operator $\mathbb A_\theta$ mapping $H$ into $L^2(\R \times \mathcal O, P_\theta)$ of the form
\begin{equation}
\mathbb A_\theta[h](x,y) = (y-u_\theta(x)) \times \mathbb I_\theta(h)(x),~~y \in \R,~ x \in \mathcal O.
\end{equation}
For the present tangent space $H$ we have $\bar H = \mathbb H$. To apply the general results from Section \ref{infog} we now calculate the adjoint $\mathbb I_\theta^*: L_\lambda^2(\mathcal O) \to \bar H= L_\lambda^2(\mathcal O)$ of $\mathbb I_\theta: \bar H \to L^2(\mathcal O)$. 

\begin{proposition}\label{adjform}
The adjoint $\mathbb I_\theta^*: L^2_\lambda(\mathcal O) \to L^2_\lambda(\mathcal O)$ of $\mathbb I_\theta$ is given by 
\begin{equation} \label{adjscop}
\mathbb I_\theta^*[g] = \nabla u_\theta \cdot \nabla V_\theta[g],~~ g \in L_\lambda^2(\mathcal O).
\end{equation}
\end{proposition}
\begin{proof}
Since $\mathbb I_\theta$ from (\ref{redscop}) defines a bounded linear operator on the Hilbert space $L^2_\lambda=\mathbb H$, a unique adjoint operator $I^*_\theta$ exists by the Riesz-representation theorem. Let us first show that 
\begin{equation}\label{trivid}
\langle h, (I^*_\theta-\mathbb I_\theta^*)g \rangle_{L^2}=0,~~\forall h, g \in C^\infty_0(\mathcal O).
\end{equation}
Indeed, since $V_\theta$ is self-adjoint for $L^2_\lambda$ and satisfies $[V_\theta g]_{|\partial \mathcal O} = 0$, we can apply the divergence theorem (Prop.~2.3 on p.143 in \cite{Tay11}) with vector field $X=h \nabla u_\theta$ to deduce
\begin{align*}
\langle h, I^*_\theta g \rangle_{L^2(\mathcal O)}&=\langle \mathbb I_\theta h, g \rangle_{L^2(\mathcal O)}=-\langle V_\theta  [\nabla \cdot (h \nabla u_\theta)] , g \rangle_{L^2(\mathcal O)} \\
& = -\int_\mathcal O  [\nabla \cdot (h \nabla u_\theta)] V_\theta [g] d\lambda  \\
& = \int_\mathcal O h \nabla u_\theta \cdot \nabla V_\theta[g] d\lambda = \langle h, \mathbb I^*_\theta g \rangle_{L^2(\mathcal O)},
\end{align*}
so that (\ref{trivid}) follows. Since $C^\infty_0(\mathcal O)$ is dense in $L^2_\lambda(\mathcal O)$ and since $I_\theta^*, \mathbb I_\theta^*$ are continuous on $L^2_\lambda(\mathcal O)$ (by construction in the former case and by (\ref{ellipt}), $u_\theta \in C^\infty(\mathcal O)$, in the latter case), the identity (\ref{trivid}) extends to all $g \in L^2_\lambda(\mathcal O)$ and hence $I^*_\theta = \mathbb I_\theta^*$, as desired.
\end{proof}

Note further that for $\theta \in \Theta$ fixed, using (\ref{ellipt}), $u_\theta \in C^\infty$ and $L^2$-continuity of $\mathbb I_\theta$, we have $\|\mathbb I_\theta^* \mathbb I_\theta h\|_{H^1} \lesssim \|\mathbb I_\theta h\|_{L^2} \lesssim \|h\|_{L^2}$.  The compactness of the embedding $H^1 \subset L^2$ now implies that the information operator $\mathbb I_\theta^* \mathbb I_\theta$ is a compact and self-adjoint operator on $L^2(\mathcal O)$.

\subsection{Injectivity of $\mathbb I_\theta$, $\mathbb I_\theta^* \mathbb I_\theta$}

Following the developments in Section \ref{infog}, our ultimate goal is to understand the range $R(\mathbb I_\theta^*)$ of the adjoint operator $\mathbb I_\theta^*$. A standard Hilbert space duality argument implies that 
\begin{equation}\label{kerran}
R(\mathbb I_\theta^*)^{\perp}=\text{ker}(\mathbb I_\theta),
\end{equation}
that is, the ortho-complement (in $\mathbb H$) of the range of $\mathbb I_\theta^*$ equals the kernel (null space) of $\mathbb I_\theta$ (in $\mathbb H$). Thus if $\psi$ is in the kernel of $\mathbb I_\theta$ then it cannot lie in the range of the adjoint and the non-existence of the inverse Fisher information in Theorem \ref{aad} for such $\psi$ can be attributed simply to the lack of injectivity of $\mathbb I_\theta$. 

We first show that under the natural `global identification' condition (\ref{lowbd}), the mapping $\mathbb I_\theta$ from (\ref{redscop}) is injective on the tangent space $H$ (and hence on our parameter space $\Theta$). The proof (which is postponed to Section \ref{injprf}) also implies injectivity of the information operator $\mathbb I_\theta^* \mathbb I_\theta$ on $H$, and in fact gives a $H^2-L^2$ Lipschitz stability estimate for $\mathbb I_\theta$.

\begin{theorem}\label{inject}
In the setting of Theorem \ref{lanell}, suppose also that (\ref{lowbd}) holds true. Then for $\mathbb I_\theta$ from (\ref{redscop}), every $\theta \in \Theta$ and some $c=c(\mu, c_0, \theta, \mathcal O)$
\begin{equation}\label{gradstab}
\|\mathbb I_\theta[h]\|_{H^2} \ge c \|h\|_{L^2} ~~\forall h \in H.
\end{equation}
In particular $\mathbb I_\theta(h)=0$ or $\mathbb I_\theta^* \mathbb I_\theta (h)=0$ imply $h=0$ for all $h\in H$.
\end{theorem}

\smallskip

Using (\ref{ellipt}) one shows further that the operator $\mathbb I_\theta$ is continuous from $H^1(\mathcal O) \to H^2(\mathcal O)$ and by taking limits in (\ref{gradstab}), Theorem \ref{inject} then extends to all $h \in H^1_0(\mathcal O)$ obtained as the completion of $H$ for the $H^1(\mathcal O)$-Sobolev norm. 

Of course, the kernel in (\ref{kerran}) is calculated on the Hilbert space $\mathbb H=L^2(\mathcal O)$, so the previous theorem does not characterise $R(\mathbb I_\theta^*)^{\perp}$ yet. Whether $\mathbb I_\theta$ is injective on all of $L^2(\mathcal O)$ depends on finer details of the PDE (\ref{PDE}). Let us illustrate this in the model examples from above.

\subsubsection{Example \ref{ex1} continued; on the kernel in $L^2(\mathcal O)$.} 

\smallskip

In our first example, $\mathbb I_\theta$ starts to have a kernel already when $h_{|\partial \mathcal O}\neq 0$. Indeed, from the proof of Theorem \ref{inject}, a function $h \in C^\infty(\bar {\mathcal O})$ is in the kernel of $\mathbb I_\theta$ if and only if 
\begin{equation}\label{tinj}
T_\theta(h)=\nabla \cdot (h \nabla u_\theta) = \nabla h \cdot \nabla u_\theta + h \Delta u_\theta=0.
\end{equation}
Now fix any $\theta \in \Theta$ with $u_\theta$ satisfying (\ref{gradlb}). The integral curves $\gamma(t)$ in $\mathcal O$ associated to the smooth vector field $\nabla u_\theta \neq 0$ are given near $x \in \mathcal O$ as the unique solutions (e.g., \cite{Tay11}, p.9) of the vector ODE 
\begin{equation} \label{orbit}
\frac{d \gamma}{dt} = \nabla u_\theta (\gamma),~~\gamma(0)=x.
\end{equation}
Since $\nabla u_\theta$ does not vanish we obtain through each $x \in \mathcal O$ a unique curve $(\gamma(t): 0 \le t \le T_\gamma)$ originating and terminating at the boundary $\partial \mathcal O$, with finite `travel time' $T_\gamma \le T(\mathcal O, c_\nabla)<\infty$. Along this curve, (\ref{tinj}) becomes the ODE
$$\frac{d}{dt}h(\gamma(t)) + h(\gamma(t)) \Delta u_\theta(\gamma(t)) =0,~~0 < t <T_\gamma.$$
Under the constraint $h|_{\partial \mathcal O}=0$ for $h \in H$, the unique solution of this ODE is $h=0$, which is in line with Theorem \ref{inject}. But for other boundary values of $h$, non-zero solutions exist. One can characterise the elements $h \in C^\infty(\bar {\mathcal O})$ in the kernel of $\mathbb I_{\theta}$ as follows. Since the vector field $\nabla u_{\theta}$ is non-trapping there exists (\cite[Theorem 6.4.1]{DH_72}) $r\in C^{\infty}(\bar{\mathcal O})$ such that $\nabla u_{\theta} \cdot \nabla r=\Delta u_{\theta}$.
Thus
\[\nabla u_{\theta} \cdot \nabla (he^{r})=e^{r}T_{\theta}(h)\]
and it follows that $T_{\theta}(h)=0$ iff $he^{r}$ is a first integral of $\nabla u_{\theta}$. Observe that the set of first integrals of $\nabla u_{\theta}$ is rather large: using the flow of $\nabla u_{\theta}$ we can pick coordinates
$(x_{1},\dots,x_{d})$ in $\mathcal O$ such that $t\mapsto (t+x_{1},x_{2},\dots,x_{d})$ are the integral curves of $\nabla u_{\theta}$ and thus any function that depends only on $x_{2},\dots,x_{d}$ is a first integral.

\smallskip

\subsubsection{Example \ref{ex2} continued; injectivity on $L^2(\mathcal O)$.} 

\smallskip

We now show that in the context of Example \ref{ex2}, the injectivity part of Theorem \ref{inject} \textit{does} extend to all of $L^2(\mathcal O)$. 

\begin{proposition}\label{dynam}
Let $\mathbb I_\theta$ be as in (\ref{redscop}) where $u_\theta$ solves (\ref{PDE}) with $f, g$ as in (\ref{repex}) and $\mathcal O$ is the unit disk in $\R^2$ centred at $(0,0)$. Then for $\theta =1$, the map $\mathbb I_1 : L^2(\mathcal O) \to L^2(\mathcal O)$ is injective.
\end{proposition}
\begin{proof}
Let us write $I=\mathbb I_1$ and suppose $I(f)=0$ for $f \in L^2(\mathcal O)$. Then for any $h \in C^\infty(\mathcal O)$ we have by Proposition \ref{adjform}
\begin{equation}
0=\langle I f, h \rangle_{L^2(\mathcal O)} = \langle f, I^* h \rangle_{L^2(\mathcal O)} = \langle f, X V_1 [h] \rangle_{L^2(\mathcal O)}
\end{equation}
with vector field $X=\nabla u_1 \cdot \nabla (\cdot)= x_1 \partial x_1 + x_2 \partial x_2$, $(x_1,x_2) \in \mathcal O$. Choosing $h=\Delta g$ for any smooth $g$ of compact support we deduce that 
\begin{equation}\label{eq:i=0}
\int_\mathcal O X(g) f d\lambda=0,~~ \forall g \in C^\infty_0(\mathcal O),
\end{equation}
and we now show that this implies $f=0$. A somewhat informal dynamical argument would say that \eqref{eq:i=0} asserts that $fd\lambda$ is an invariant density under the flow of $X$. Since the flow of $X$ in backward time has a sink at the origin, the density can only be supported at $(x_1,x_2)=0$ and thus $f=0$.

One can give a distributional argument as follows. Suppose we consider polar coordinates $(r,\vartheta)\in (0,1)\times S^{1}$ and functions $g$ of the form $\phi(r)\psi(\vartheta)$, where $\phi\in C_{0}^{\infty}(0,1)$ and $\psi\in C^{\infty}(S^{1})$. In polar coordinates $X=r\partial_{r}$ and hence we may write \eqref{eq:i=0} as
\begin{equation}
\int_{0}^{1}\left(r^{2}\left(\int_{0}^{2\pi}f(r,\vartheta)\psi(\vartheta)\,d\vartheta\right)\partial_{r}\phi\right)\,dr=0.
\label{eq:i=02}
\end{equation}
By Fubini's theorem, for each $\psi$ we have an integrable function
\[F_{\psi}(r):=\int_{0}^{2\pi}f(r,\vartheta)\psi(\vartheta)\,d\vartheta\]
and thus $r^{2}F_{\psi}$ defines an integrable function on $(0,1)$ whose distributional derivative satisfies
$\partial_{r}(r^{2}F_{\psi})=0$ by virtue of \eqref{eq:i=02}. Thus $r^{2}F_{\psi}=c_{\psi}$ (using that a distribution on $(0,1)$ with zero derivative must be a constant).
Now consider $\psi\in C^{\infty}(S^{1})$ also as a function in $L^{2}(\mathcal O)$ and compute the pairing
\[(f,\psi)_{L^{2}(\mathcal O)}=\int_{0}^{1}rF_{\psi}(r)\,dr=c_{\psi}\int_{0}^{1}r^{-1}\,dr=\pm\infty\]
unless $c_{\psi}=0$. Thus $f=0$.

\end{proof}

By perturbation (similar as in (\ref{pertubio})) and the Morse lemma, we can show that $u_\theta, \theta \in \Theta,$ has a gradient $u_\theta$ that vanishes only at a single point in a neighbourhood of $0$, and so the proof of the previous theorem extends to any $\theta \in \Theta$.

 \subsection{The range of $\mathbb I_\theta^*$ and transport PDEs} \label{surj}
 
From (\ref{kerran}) we see $\overline{R(\mathbb I_\theta^*)}= \text{ker}(\mathbb I_\theta)^{\perp},$ but in our infinite-dimensional setting care needs to be exercised as the last identity holds in the (complete) Hilbert space $\mathbb H = L^2(\mathcal O)$ rather than in our tangent space $H$ (on which the kernel of $\mathbb I_\theta$ is trivial). We will now show that the range $R(\mathbb I_\theta^*)$ remains strongly constrained. This is also true in Example \ref{ex2} when $\text{ker}(\mathbb I_\theta)=\{0\}$: the range may not be closed $\overline{R(\mathbb I_\theta^*)}\neq R(\mathbb I_\theta^*)$, and this `gap' can be essential in the context of Theorems \ref{aad} and \ref{n00}. To understand this, note that from Proposition \ref{adjform} we have
\begin{equation} \label{range}
R(\mathbb I_\theta^*)=\big\{\psi =  \nabla u_\theta \cdot \nabla  V_\theta [g],~\text{for some } g \in L^2_\lambda(\mathcal O)\big\},
\end{equation}
The operator $V_\theta$ maps $L^2_\lambda$ into $H^2_0=\{y \in H^2: y_{|\partial \mathcal O}=0\}$ and hence if $\psi$ is in the range of $\mathbb I_\theta^*$ then the equation
\begin{align}\label{transport}
\nabla u_\theta \cdot \nabla y = \psi~~\text{on } \mathcal O \\
y=0 \text{ on } \partial \mathcal O \notag
\end{align}
necessarily has a solution $y=y_\psi \in H^2_0$. The existence of solutions to the transport PDE (\ref{transport})  depends crucially on the compatibility of $\psi$ with geometric properties of the vector field $\nabla u_\theta$, which in turn is determined by the geometry of the forward map $\G$ (via $f,g,\theta$) in the base PDE (\ref{PDE}). We now illustrate this in our two model Examples \ref{ex1} and \ref{ex2}.

\smallskip

\subsubsection{Example \ref{ex1} continued; range constraint.} 

\smallskip

Applying the chain rule to $y \in H^2(\mathcal O)$ and using (\ref{orbit}) we see
$$ \frac{d}{dt} y(\gamma(t)) = \frac{d \gamma (t)}{dt} \cdot \nabla y(\gamma(t)) = (\nabla u_\theta  \cdot \nabla y)(\gamma(t)),~~0 < t <T_\gamma.$$
Hence along any integral curve $\gamma$ of the vector field $\nabla u_\theta$, the PDE (\ref{transport}) reduces to the ODE
\begin{equation}\label{oder}
\frac{dy}{dt} = \psi.
\end{equation}
Now suppose $\psi \in R(\mathbb I_\theta^*)$ then a solution $y \in H^2_0$ to (\ref{transport}) satisfying $y_{|\partial \mathcal O}=0$ must exist. Such $y$ then also solves the ODE (\ref{oder}) along each curve $\gamma$, with initial and terminal values $y(0)=y(T_\gamma)=0$. By the fundamental theorem of calculus (and uniqueness of solutions) this forces 
\begin{equation}\label{ftc}
\int_0^{T_\gamma} \psi(\gamma(t))dt  =0
\end{equation}
to vanish. In other words, $\psi$ permits a solution $y$ to (\ref{transport}) only if $\psi$ integrates to zero along each integral curve (orbit) induced by the vector field $\nabla u_\theta$. Now consider any smooth (non-zero) \textit{nonnegative} $\psi$ in the tangent space $H=C^\infty_0(\mathcal O)$, and take $x \in \mathcal O$ such that $\psi \ge c>0$ near $x$. For $\gamma$ the integral curve passing through $x$ we then cannot have (\ref{ftc}) as the integrand never takes negative values while it is positive and continuous near $x$. Conclude by way of contradiction that $\psi \notin R(\mathbb I_\theta^*)$.  Applying Theorems \ref{aad} and \ref{n00}, we have proved: 
\begin{theorem}\label{nogood}
Consider estimation of the functional $\Psi(\theta)=\langle \theta, \psi \rangle_{L^2(\mathcal O)}$ from data $(Y_i, X_i)_{i=1}^N$ drawn i.i.d.~from $P_\theta^N$ in the model (\ref{modelp}) where $f,g$ in (\ref{PDE}) are chosen as in (\ref{repex}), the domain $\mathcal O$ is separated away from the origin, and $\Theta$ is as in (\ref{bigtheta}) with $\eta$ small enough and $\beta>1+d, d \le 3$. Suppose $0 \neq \psi \in C^\infty_0(\mathcal O)$ satisfies $ \psi \ge 0$ on $\mathcal O$. Then for every $\theta \in \Theta$ the efficient Fisher information for estimating $\Psi(\theta)$ satisfies
\begin{equation}\label{nofish}
\inf_{h \in H, \langle h, \psi \rangle_{L^2}\neq 0}\frac{\|\mathbb I_\theta h\|_{L^2_\lambda}^2}{\langle \psi, h \rangle_{L^2_\lambda}^2} =0.
\end{equation}
In particular, for any $\theta \in \Theta$,
\begin{equation}\label{noway}
\liminf_{N \to \infty} \inf_{\tilde \psi_N: (\R  \times \mathcal O)^N \to \R} \sup_{\theta' \in \Theta, \|\theta'-\theta\|_{\mathbb H} \le 1/\sqrt N} N E_{\theta'}^N(\tilde \psi_N - \Psi(\theta'))^2 = \infty.
\end{equation}
\end{theorem}

Let us notice that one can further show that (\ref{ftc}) is also a \textit{sufficient} condition for $\psi$ to lie in the range of $\mathbb I_\theta^*$ (provided $\psi$ is smooth and with compact support in $\mathcal O$). As this condition strongly depends on $\theta$ via the vector field $\nabla u_\theta$, it seems difficult to describe any choices of $\psi$ that lie in $\cap_{\theta \in \Theta} R(\mathbb I_\theta^*)$. 

\smallskip

\subsubsection{Example \ref{ex2} continued; range constraint}

\smallskip

We showed in the setting of Example \ref{ex2} that $\mathbb I_\theta$ is injective on all of $L^2(\mathcal O)$, and hence any $\psi \in L^2(\mathcal O)$ lies in \textit{closure} of the range of $\mathbb I_\theta^*$. Nevertheless, there are many relevant $\psi$'s that are not contained in $R(\mathbb I_\theta^*)$. In Example \ref{ex2}, the gradient of $u_\theta$ vanishes and the integral curves $\gamma$ associated to $\nabla u_\theta=(x_1, x_2)$ emanate along straight lines from $(0,0)$ towards boundary points $(z_1, z_2) \in \partial \mathcal O$ where $y((z_1, z_2))=0$. If we parameterise them as $\{(z_1e^{t}, z_2e^{t}): -\infty < t \le 0\}$, then as after (\ref{oder}) we see that if a solution $y \in H^2_0$ to (\ref{transport}) exists then $\psi$ must necessarily satisfy
\begin{equation}
\int_{-\infty}^0 \psi(z_1e^{t}, z_2e^{t}) dt =0 - y(0) =const.~~\forall (z_1, z_2) \in \partial \mathcal O.
\end{equation}
This again cannot happen, for example, for any non-negative non-zero $\psi \in H$ that vanishes along a given curve $\gamma$ (for instance if it is zero in any given quadrant of $\mathcal O$), as this forces $const=0$. Theorems \ref{aad} and \ref{n00} again yield the following for Example \ref{ex2}:

\begin{theorem}\label{nogood2}
Consider the setting of Theorem \ref{nogood} but where now $\mathcal O$ is the unit disk centred at $(0,0)$, and where $0 \le \psi \in C^\infty_0(\mathcal O), \psi \neq 0,$ vanishes along some straight ray from $(0,0)$ to the boundary $\partial \mathcal O$. Then (\ref{nofish}) and (\ref{noway}) hold at $\theta=1$.
\end{theorem}

Arguing as after Proposition \ref{dynam}, the result can be extended to any $\theta \in \Theta$ by an application of the Morse lemma.

\subsection{Concluding remarks}

\begin{remark}\label{intrcrit}{Interior saddle points of $u_\theta$.} \normalfont To complement Examples \ref{ex1}, \ref{ex2}, suppose we take $\theta=1, f=0$ in (\ref{PDE}) so that $u=u_1=x_1^2-x_2^2$ if $g = u_{\partial \mathcal O}$ (and $\mathcal O$ is the unit disk, say). Then $\nabla u = 2(x_1,-x_2)$ and the critical point is a saddle point. In this case we can find integral curves $\gamma_x$ running through $x$ away from $(0,0)$ between boundary points in finite time. Then is $\psi$ is nonnegative and supported near $x$ it cannot integrate to zero along $\gamma_x$. An analogue of Theorem \ref{nogood} then follows for this constellation of parameters in (\ref{PDE}), too. Note that in this example, the kernel of $\mathbb I_\theta$ contains at least all constants.
\end{remark}

\begin{remark}{Local curvature of $\G$. } \normalfont
The quantitative nature of (\ref{gradstab}) in Theorem \ref{inject} is compatible with `gradient stability conditions' employed in \cite{NW20, BN21} to establish polynomial time posterior computation time bounds for gradient based Langevin MCMC schemes. Specifically, arguing as in Lemma 4.7 in \cite{NW20}, for a neighbourhood $\mathcal B$ of $\theta_0$ one can deduce local average `curvature' $$\inf_{\theta \in \mathcal B}\lambda_{min} E_{\theta_0}[-\nabla^2 \ell (\theta)] \ge c_2 D^{-4/d},$$ of the average-log-likelihood function $\ell$ when the model $\Theta$ is discretised in the eigen-basis $E_D \equiv (e_n: n \le D) \subset H$ arising from the Dirichlet Laplacian. In this sense (using also the results from \cite{GN20}) one can expect a Bayesian inference method based on data (\ref{model}) and Gaussian process priors to be consistent and computable even in high-dimensional settings. This shows that such local curvature results are not sufficient to establish (and hence distinct from) Gaussian `Bernstein-von Mises-type' approximations.
\end{remark}

\begin{remark}\label{normald}{Boundary constraints on $\theta$.}  \normalfont
As the main flavour of our results is `negative', the assumption of knowledge of the boundary values of $\theta$ in (\ref{bigtheta}) strengthens our conclusions -- it is also natural as the regression function $u=g$ is already assumed to be known at $\partial \mathcal O$. In the definition of the parameter space $\Theta$ we could further have assumed that all outward normal derivatives up to order $\beta-1$ of $\theta$ vanish at $\partial \mathcal O$. This would be in line with the parameter spaces from \cite{GN20, NVW18}. All results in this section remain valid  because our choice of tangent space $H$ in (\ref{tangentspace}) is compatible with this more constrained parameter space. \end{remark}

\begin{remark}\label{why}{Ellipticity.} \normalfont
The Bernstein-von Mises theorems from \cite{N17, MNP20, MNP17} exploit \textit{ellipticity} of the information operator $\mathbb I_\theta^* \mathbb I_\theta$ in their settings, allowing one to solve for $y$ in the equation $\mathbb I_\theta^* \mathbb I_\theta y =\psi$ so that $R(\mathbb I_\theta^*)$ contains at least all smooth compactly supported $\psi$ (and this is so for \textit{any} parameter $\theta \in \Theta$). In contrast, in the present inverse problem arising from (\ref{PDE}), the information operator does not have this property and solutions $y$ to the critical equation $\mathbb I^*_\theta y = \psi$ exist only under stringent geometric conditions on $\psi$. Moreover, these conditions exhibit a delicate dependence on $\theta$, further constraining the set $\cap_{\theta \in \Theta} R(\mathbb I_\theta^*)$ relevant for purposes of statistical inference. 
\end{remark}


\section{Appendix}

For convenience of the reader we include here a few more proofs of some results of this article.

\subsection{Proofs of Theorem \ref{inject} and Proposition \ref{oldboy}}\label{injprf}

Define the operator $$T_\theta(h) = \nabla \cdot (h \nabla u_\theta),~~ h \in H,$$  so that (\ref{redscop}) becomes $\mathbb I_\theta = V_\theta \circ T_\theta$. The map $u \mapsto (\mathcal L_\theta u, u_{|\partial \mathcal O})$ is a topological isomorphism between $H^{2}(\mathcal O)$ and $L^2(\mathcal O) \times H^{3/2}(\partial \mathcal O)$, (\cite{LM72}, Theorem II.5.4), and hence with $u=V_\theta[w]$ we deduce $\|V_\theta [w]\|_{H^2} \gtrsim \|w\|_{L^2}$ for all $w \in C^\infty(\mathcal O)$. As a consequence, using also Lemma \ref{coerc},
\begin{equation*}
\|\mathbb I_\theta[h]\|_{H^2} \gtrsim \|T_\theta(h)\|_{L^2} \gtrsim \|h\|_{L^2},~~~h \in H
\end{equation*}
which proves the inequality in Theorem \ref{inject}. Next, as $\mathbb I_\theta$ is linear we see that whenever $\mathbb I_\theta[h_1]= \mathbb I_\theta[h_2]$ for $h_1, h_2 \in H$ we have $\mathbb I_{\theta}[h_1-h_2]=0$ and so by the preceding inequality $h=h_1-h_2=0$ in $L^2$, too.  Likewise if $h_1, h_2 \in H$ are such that $\mathbb I_\theta^* \mathbb I_\theta h_1 = \mathbb I_\theta^* \mathbb I_\theta h_2$, then $0 = \langle \mathbb I_\theta^* \mathbb I_\theta (h_1 - h_2), h_1 -h_2 \rangle_{L^2_\lambda} = \|\mathbb I_\theta (h_1-h_2)\|_{L_\lambda^2}^2$
so $\mathbb I_\theta h_1 = \mathbb I_\theta h_2$ and thus by what precedes $h_1=h_2$. 

\smallskip

\begin{lemma}\label{coerc}
We have $\|T_\theta(h)\|_{L^2} =\| \nabla \cdot (h \nabla u_\theta)\|_{L^2} \ge c \|h\|_{L^2}$ for all $h \in H$ and some constant $c=c(\mu, B, c_0)>0$, where $B \ge \|u_\theta\|_\infty$.
\end{lemma}
\begin{proof}
Applying the Gauss-Green theorem to any $v \in C^1(\mathcal O)$ vanishing at $\partial \mathcal O$  gives
		\begin{align*}
		\langle \Delta u_{\theta}, v^2 \rangle_{L^2} + \frac{1}{2} \langle \nabla u_{\theta}, \nabla (v^2) \rangle_{L^2} = \frac{1}{2}\langle \Delta u_{\theta}, v^2 \rangle_{L^2}.
		\end{align*}
		For $v=e^{-\mu u_{\theta}}h, h \in H,$ with $\mu>0$ to be chosen we thus have
		\begin{equation*}
		\frac{1}{2} \int_{\mathcal O}  \nabla (v^2) \cdot \nabla u_{\theta} =- \int_{\mathcal O} \mu \|\nabla u_{\theta}\|^2 v^2 + \int_\mathcal O v e^{-\mu u_{\theta}}\nabla h \cdot \nabla u_{\theta},
		\end{equation*}
		so that by the Cauchy-Schwarz inequality
		\begin{align} \label{keylb}
		 \left|\int_{\mathcal O}\Big(\frac{1}{2}\Delta u_{\theta}+\mu \|\nabla u_{\theta}\|^2\Big)v^2 \right| &= \left|\langle (\Delta u_{\theta} + \mu \|\nabla u_{\theta}\|^2), v^2 \rangle_{L^2} + \frac{1}{2} \langle \nabla u_{\theta}, \nabla (v^2)\rangle_{L^2}\right|  \\
		&= \left|\langle h \Delta u_{\theta} + \nabla h \cdot \nabla u_{\theta}, h e^{-2\mu u_{\theta}} \rangle_{L^2} \right| \le \mu \|\nabla \cdot (h \nabla u_{\theta})\|_{L^2}  \|h\|_{L^2} \notag
		\end{align}
		for $\bar \mu=\exp(2\mu \|u_{\theta}\|_{\infty})$.  We next lower bound the multipliers of $v^2$ in l.h.s.~of (\ref{keylb}): By (\ref{lowbd})
		$$ \left|\int_{\mathcal O}\Big(\frac{1}{2}\Delta u_{\theta}+\mu \|\nabla u_{\theta}\|^2\Big)v^2 \right|   \ge c_0 \int_{\mathcal O} v^2$$ and combining this with (\ref{keylb}) we deduce
		\begin{equation*}
		\|\nabla \cdot (h \nabla u_{\theta})\|_{L^2}\|h\|_{L^2}\geq c'\|v\|^2_{L^2(\mathcal O)}\gtrsim \|h\|^2_{L^2},~h \in H,
		\end{equation*}
		which is the desired estimate.
\end{proof}
The last lemma also immediately implies Proposition \ref{oldboy}: Let us write $h=\theta_1-\theta_2$ which defines an element of $H$. Then by (\ref{PDE}) we have
$\nabla \cdot (h \nabla u_{\theta_1}) = \nabla \cdot (\theta_2 \nabla (u_{\theta_2}-u_{\theta_1}))$
and hence 
$\|\nabla \cdot (h \nabla u_{\theta_1})\|_{L^2} \lesssim  \|u_{\theta_2}-u_{\theta_1}\|_{H^2}.$ By Lemma \ref{coerc} the l.h.s.~is lower bounded by a constant multiple of $\|h\|_{L^2}= \|\theta_1-\theta_2\|_{L^2}$, so that the result follows.

\subsection{Proof of Theorem \ref{aad} for $\mathbb I_\theta^* \mathbb I_\theta$ compact} \label{spec}

Let us assume $\bar H = \mathbb H$ without loss of generality, write $I \equiv\mathbb I_\theta, L^2=L^2_\lambda(\mathcal X)$ in this proof, and let $\text{ker}(I^*I)=\{h \in \mathbb H: I^*Ih=0\}$. If $I^*I$ is a compact operator on $\mathbb H$ then by the spectral theorem for self-adjoint operators, there exists an orthonormal system of $\mathbb H$ of eigenvectors $\{e_k: k \in \mathbb N\}$ spanning $\mathbb H \ominus \text{ker}(I^*I)$ corresponding to eigenvalues $\lambda_k > 0$ so that $$I^*Ie_k = \lambda_k e_k,~~\text{and }I^*I h = \sum_k \lambda_k \langle h, e_k \rangle_\mathbb H e_k,~~ h \in \mathbb H.$$ We can then define the usual square root operator $(I^*I)^{1/2}$ by
\begin{equation}\label{root}
(I^*I)^{1/2} h = \sum_k \lambda^{1/2}_k \langle h, e_k \rangle_\mathbb H e_k,~~ h \in \mathbb H.
\end{equation}
If we denote by $P_0$ the $\mathbb H$-projection onto $\text{ker}(I^*I)$, then the range of $(I^*I)^{1/2}$ equals
\begin{equation}\label{rootrange}
R((I^*I)^{1/2}) = \Big\{g \in \mathbb H: P_0(g)=0, \sum_{k} \lambda_k^{-1} \langle e_k, g \rangle_{\mathbb H}^2 <\infty \Big\}.
\end{equation}
Indeed using standard Hilbert space arguments, a) since $P_0(e_k)=0$ for all $k$, for any $h \in \mathbb H$ the element $g=(I^*I)^{1/2} h$ belongs to the right hand side in the last display, and conversely b) if $g$ satisfies $P_0(g)=0$ and $\sum_{k} \lambda_k^{-1} \langle e_k, g \rangle_\mathbb H^2<\infty$ then $h=\sum_k \lambda_k^{-1/2} \langle e_k, g\rangle e_k$ belongs to $\mathbb H$ and $(I^*I)^{1/2} h = g$. 

Next, Lemma A.3 in \cite{vdV91} implies that $R(I^*)=R((I^*I)^{1/2})$. Now suppose $\psi \in \mathbb H$ is such that $\psi \notin R(I^*)$ and hence $\psi \notin R((I^*I)^{1/2})$. Then from (\ref{rootrange}), either $P_0(\psi)\neq 0$ or $\sum_{k} \lambda_k^{-1} \langle e_k, \psi \rangle_{\mathbb H}^2 =\infty$ (or both). In the first case, let $\bar h = P_0(\psi)$ so
$$\|I\bar h\|_{L^2}=\|I(P_0(\psi))\|_{L^2}= \langle I^*I(P_0(\psi)), P_0(\psi) \rangle_\mathbb H=0$$  but $\langle \psi, \bar h\rangle_{\mathbb H} = \|P_0 \psi\|_{\mathbb H}^2= \delta$ for some $\delta>0$. Since $H$ is dense in $\mathbb H$, for any $\epsilon, 0<\epsilon<\min(\delta/(2\|\psi\|_\mathbb H), \delta^2/4),$ we can find $h \in H$ such that $\|h-\bar h\|_{\mathbb H}<\epsilon$ and by continuity also $\|I(h-\bar h)\|_{L^2}<\epsilon$. Then 
$$\sqrt{i_{\theta, h, \psi}} = \frac{\|Ih\|_{L^2}}{|\langle \psi, h \rangle_{\mathbb H}|} \le 2\frac{\epsilon}{\delta} \le \sqrt \epsilon.$$ Using also (\ref{lanip}) we conclude that $i_{\theta, H, \psi}<\epsilon$ in (\ref{efffish}), so that the result follows since $\epsilon$ was arbitrary. In the second case we have $\sum_{k} \lambda_k^{-1} \langle e_k, \psi \rangle_{\mathbb H}^2 =\infty$ and define $$\psi_N = \sum_{k \le N} \lambda_{k}^{-1} e_k \langle e_k, \psi \rangle_\mathbb H,~~N \in \mathbb N,$$ which defines an element of $\mathbb H$. By density we can choose $h_N \in H$ such that $\|h_N - \psi_N\|_{\mathbb H} <1/\|\psi\|_{\mathbb H}$ as well as $\|I(h_N-\psi_N)\|_{L^2}<1$, for every $N$ fixed. Next observe that
\begin{equation*}
\langle \psi, \psi_N \rangle_{\mathbb H} = \sum_{k \le N} \lambda_k^{-1} \langle e_k, \psi \rangle_{\mathbb H}^2 \equiv M_N
\end{equation*}
\begin{equation*}
\|I(\psi_N)\|^2_{L^2} = \langle I^*I(\psi_N), \psi_N \rangle_{\mathbb H} = \sum_{k \le N} \lambda_k^{-1} \langle e_k, \psi \rangle_{\mathbb H}^2 =M_N
\end{equation*}
and that $M_N \to \infty$ as $N \to \infty$. Then by our choice of $h_N \in H$ and if $M_N \ge 2$ we have by the triangle inequality,
$$|\langle \psi, h_N \rangle_{\mathbb H}| \ge |\langle \psi, \psi_N \rangle_{\mathbb H}| -  |\langle \psi, \psi_N -h_N\rangle_{\mathbb H}| \ge M_N - 1 \ge M_N/2,$$
$$\|I(h_N)\|_{L^2} \le \|I(\psi_N)\|_{L^2} + \|I(h_N - \psi_N)\|_{L^2} \le \sqrt M_N+1 \le 2 \sqrt M_N.$$ From this and (\ref{lanip}) we conclude that the inverse of (\ref{efffish}) satisfies
$$i^{-1}_{\theta, H, \psi} \ge \frac{\langle \psi, h_N \rangle_{\mathbb H}^2}{\|Ih_N\|_{L^2}^2} \ge \frac{1}{16} \frac{M_N^2}{M_N} \ge M_N/16.$$ As $N$ was arbitrary and $M_N \to_{N\to \infty} \infty$ we must have $i_{\theta, H, \psi}=0$, as desired.

\medskip

\textbf{Acknowledgement.} We are grateful to Jan Bohr and Lauri Oksanen for helpful remarks and discussions.

\bibliographystyle{plain}
\bibliography{ref.bib}

\end{document}